\documentclass[10pt]{article}
\usepackage{amsfonts}

\usepackage{amsmath}
\usepackage{graphicx}

\newtheorem{theorem}{Theorem}[section]

\newtheorem{problem}[theorem]{Problem}
\newtheorem{proposition}[theorem]{Proposition}
\newtheorem{remark}[theorem]{Remark}

\newenvironment{proof}[1][Proof]{\textbf{#1.} }{\ \rule{0.5em}{0.5em}}

\begin{document}

\title{ A singular stochastic differential equation driven by
fractional Brownian motion
 }
\author{Yaozhong Hu\thanks{
Y. Hu is supported  by the NSF  grant
 DMS0504783.}, David Nualart\thanks{  D. Nualart is  supported    by
  the NSF grant   DMS0604207.}, Xiaoming Song \\
%EndAName
Department of Mathematics \\
University of Kansas \\
Lawrence, Kansas, 66045 USA}

\date{}
\maketitle
\begin{abstract} In this paper we
study a singular stochastic differential equation driven
by an additive  fractional Brownian motion with Hurst parameter $H>\frac 12$.
Under some assumptions on the drift, we show that there
is a unique solution, which has moments of all orders.
We also apply the techniques of Malliavin calculus to
prove that the solution has an absolutely continuous law at any time $t>0$.
\end{abstract}

\section{Introduction}

The aim of this paper is to study a   stochastic differential
equation, driven by an additive fractional Brownian motion (fBm) with Hurst
parameter $H>\frac{1}{2}$, assuming that the drift $f(t,x)$ has a
singularity at $x=0$ of the form $x^{-\alpha }$, where $\alpha >\frac{1}{H}-1
$. \

The study of this type of equations is partially  motivated by the
 equation  satisfied by the $d$-dimensional
fractional Bessel process $R_{t}=|B_{t}^{H}|$, $d\geq 2$ (see Guerra and
Nualart \cite{GN}, and Hu and Nualart \cite{HN}):%
\begin{equation*}
R_{t}=Y_{t}+H(d-1)\int_{0}^{t}\frac{s^{2H-1}}{R_{s}}ds,
\end{equation*}%
where the process $Y_{t}$ is equal to a divergence integral, $%
Y_{t}=\int_{0}^{t}\sum_{i=1}^{d}\frac{B_{s}^{H,i}}{R_{s}}\delta B_{s}^{H,i}$%
. \ The process $Y$ is not a one-dimensional fractional Brownian motion (see
Eisenbaum and Tudor \cite{ET} and Hu and Nualart \cite{HN} for some
results in this direction), although it shares with the fBm similar properties of scaling
and  $\frac 1H$-variation. Notice that here the initial condition is zero.
 
We are considering the case where the initial condition  $x_{0}$ is
strictly positive. Using arguments
based on fractional calculus  inspired by the estimates obtained by Hu and
Nualart in \cite{HN2}, we will show that there exist a unique
global solution which has moments of all
orders, and even negative moments, in the particular case $f(t,x)=Kx^{-1}$,
if $t$ is small enough.   
 We will also  show that the solution has an
absolutely continuous law with respect to the Lebesgue measure, using the
techniques of Malliavin calculus for the fractional Brownian motion.
As an application we obtain the existence of a
unique solution with moments of all orders for a fractional version 
of   the  CIR model in mathematical finance (\cite{CIR}), which is a singular 
stochastic differential equation driven by fractional Brownian motion 
with the diffusion coefficient being $\sqrt x$.

The paper is organized as follows. In the first section we will consider the
case of a deterministic differential equation driven by a H\"{o}lder
continuous function, and with singular drift. The case of the factional
Brownian motion is developed in Section \ 3.

\setcounter{equation}{0}

\section{Singular equations driven by rough paths}

For any $s\leq t$, $C([s,t])$ denotes the Banach space of continuous
functions equipped with the supremum norm, and $C^{\beta }([s,t])$ denotes
the space of H\"{o}lder continuous functions of order $\beta $ on $[s,t]$.
For any $x\in C([s,t])$ we put
\begin{equation*}
\Vert x\Vert _{s,t,\infty }=\sup \{|x(r)|,s\leq r\leq t\}\,,\quad
\end{equation*}%
and if $x\in C^{\beta }([s,t])$ we put
\begin{equation*}
\Vert x\Vert _{s,t,\beta }=\sup \{\frac{|x(u)-x(v)|}{|u-v|^{\beta }},s\leq
u,v\leq t\}\,\,.
\end{equation*}%
Fix $\beta \in (1/2,1)$. Suppose that $\varphi :\mathbb{R}_{+}\rightarrow
\mathbb{R}$ is a function such that $\varphi (0)=0$, and $\varphi \in
C^{\beta }([0,T])$ for all $T>0$. Consider the following deterministic
differential equation driven by the rough path $\varphi $
\begin{equation}
x_{t}=x_{0}+\int_{0}^{t}f(s,x_{s})ds+\varphi (t)\,,  \label{e.2.1}
\end{equation}%
where $x_{0}>0$ is a constant. We are going to impose the following \
assumptions on the coefficient $f$:

\begin{description}
\item[(i)] $f:[0,\infty )\times (0,\infty )\rightarrow \lbrack 0,\infty )$
is a nonnegative, continuous function which has a continuous partial
derivative with respect to $x$ such that $\partial _{x}f(t,x)\leq 0$ for all
$t>0,\,x>0$.

\item[(ii)] There exists $x_{1}>0$ and $\alpha >\frac{1}{\beta }-1$ such
that $f(t,x)\geq g(t)x^{-\alpha }$, for all $t\geq 0$ and $\,x\in (0,x_{1})$%
, where $g(t)$ is a nonnegative continuous function with $g(t)>0$ for all $%
t>0$.

\item[(iii)] $f(t,x) \leq h(t)\left( 1+\frac 1x \right)$ for all $t\geq 0$ and $x>0$, where  $h(t)$ is a certain
nonnegative locally bounded function.
\end{description}

\begin{theorem}
Under the  assumptions (i)-(ii), there exists a unique solution $x_{t}$ to
equation (\ref{e.2.1}) such that $x_{t}>0$ on $[0,\infty )$.
\end{theorem}

\begin{proof}
It is easy to see that there exists a continuous local solution $x_{t}$ to
equation (\ref{e.2.1}) on some interval $[0,T)$, where $T$ satisfies $T=\inf
\{t>0:x_{t}=0\}$. Then it suffices to show that $T=\infty $. Suppose that $%
T<\infty $. Then, then $x_{t}\rightarrow 0$, as $t\uparrow T$. Since $%
\varphi \in C^{\beta }([0,T])$, there exists a constant $C>0$, such that $%
|\varphi (t)-\varphi (s)|\leq C|t-s|^{\beta }$, \ for all $s,t\in \lbrack
0,T]$. Since $x_{t}$ satisfies the equation (\ref{e.2.1}), for all $t\in
\lbrack 0,T]$ we have
\begin{equation*}
0=x_{T}=x_{t}+\int_{t}^{T}f(s,x_{s})ds+\varphi (T)-\varphi (t).
\end{equation*}%
Since $f(s,x_{s})$ is positive, for all $t\in \lbrack 0,T]$ we have
\begin{equation*}
x_{t}\leq x_{t}+\int_{t}^{T}f(s,x_{s})ds=\varphi (t)-\varphi (T)\leq
C(T-t)^{\beta }\,.
\end{equation*}%
From the assumption (ii), there exist $t_{0}\in (0,T)$ and a constant $K>0$,
such that  $g(t)\geq K$ and $x_{s}\in (0,x_{1})$ for all $t\in \lbrack
t_{0},T)$. Then, for all $t\in \lbrack t_{0},T)$ we have
\begin{equation*}
f(t,x_{t})\geq \frac{g(t)}{x_{t}{}^{\alpha }}\geq \frac{K}{x_{t}{}^{\alpha }}%
\geq \frac{K}{C^{\alpha }\left( T-t\right) ^{{\alpha }{\beta }}}\,.
\end{equation*}%
Consequently, for all $t\in \lbrack t_{0},T)$ we obtain
\begin{equation*}
\frac{K(T-t)^{1-\alpha \beta }}{C^{\alpha }(1-\alpha \beta )}=\int_{t}^{T}%
\frac{K}{C^{\alpha }(T-s)^{\alpha \beta }}ds\leq
\int_{t}^{T}f(s,x_{s})ds\leq C(T-t)^{\beta },\,
\end{equation*}%
which is a contradiction because $1-\alpha \beta -\beta <0$ and $t$
can be arbitrarily close to $T$. Therefore, $%
T=\infty $. This proves the existence of the solution for all $t$.

Now we show the uniqueness. If $x_{1,t}$ and $x_{2,t}$ are two positive
solutions to equation (\ref{e.2.1}), then
\begin{equation*}
x_{1,t}-x_{2,t}=\int_{0}^{t}[f(s,x_{1,s})-f(s,x_{2,s})]ds\,.
\end{equation*}%
Because $\partial _{x}f(t,x)\leq 0$ for all  $t>0,\,x>0$, we deduce
\begin{equation*}
(x_{1,t}-x_{2,t})^{2}=2%
\int_{0}^{t}(x_{1,s}-x_{2,s})[f(s,x_{1,s})-f(s,x_{2,s})]ds\leq 0.
\end{equation*}%
So $x_{1,t}=x_{2,t}$.\newline

Thus we conclude that there exists a unique solution $x_t$ to the equation (%
\ref{e.2.1}) such that $x_t>0$ on $[0, \infty)$.
\end{proof}

\begin{remark}
From the continuity of $x_{t}$ and $f(t,x)$ and the H\"{o}lder continuity of
$\varphi (t)$, we   obtain that for any $T>0$, $x\in C^{\beta }([0,T])$.
\end{remark}

The next result provides an estimate on the supremum norm of the solution in
terms of the H\"{o}lder norm of the driving function $\varphi $.

\begin{theorem}
\label{t1}Let the assumptions (i)-(iii) be satisfied. If $x_{t}$ is the
solution to equation(\ref{e.2.1}), then for any $\gamma >2$, and for any $%
T>0 $,
\begin{equation}
\Vert x\Vert _{0,T,\infty }\leq C_{1,{\gamma },{\beta },T}(|x_{0}|+1)\exp
\left\{ C_{2,{\gamma },{\beta },T}\left( 1+\Vert \varphi \Vert _{0,T,\beta
}^{\frac{\gamma }{\beta (\gamma -1)}}\right) \right\} \,,  \label{e2}
\end{equation}%
where $C_{1,{\gamma },{\beta },T}$ and $C_{2,{\gamma },{\beta },T}$ are
constants depending on $\beta ,{\gamma }$, $\left\| h\right\| _{0,T,\infty }$
and $T$.
\end{theorem}

\begin{proof}
Fix a time interval $[0,T]$. Let $y_{t}=x_{t}^{\gamma }$. Then the chain
rule applied to $x_{t}^{\gamma }$ yields
\begin{equation}
y_{t}=x_{0}^{\gamma }+\gamma \int_{0}^{t}f(s,y_{s}^{\frac{1}{\gamma }%
})y_{s}^{1-\frac{1}{\gamma }}ds+\gamma \int_{0}^{t}y_{s}^{1-\frac{1}{\gamma }%
}d\varphi (s).  \label{e.2.2}
\end{equation}%
The second integral in (\ref{e.2.2}) is a Riemann-Stieltjes integral (see
Young \cite{Yo}). From Assumption (iii), we have
\begin{eqnarray}
|y_{t}-y_{s}| &=&\gamma \left| \int_{s}^{t}f(u,y_{u}^{\frac{1}{\gamma }%
})y_{u}^{1-\frac{1}{\gamma }}du+\int_{s}^{t}y_{u}^{1-\frac{1}{\gamma }%
}d\varphi (u)\right|  \notag \\
&\leq &K_{T}\gamma \int_{s}^{t}\left[
y_{u}^{1-\frac{2}{\gamma }}+y_{u}^{1-\frac{1}{\gamma }}\right]
du+\gamma \left|
\int_{s}^{t}y_{u}^{1-\frac{1}{\gamma }}d\varphi (u)\right| ,  \label{f5}
\end{eqnarray}%
where $K_{T}=\sup_{t\in \lbrack 0,T]}h(t)$. Since $\gamma >2$, we have
\begin{equation}
\int_{s}^{t}y_{u}^{1-\frac{2}{\gamma }}du\leq \left[
\Vert y\Vert _{s,t,\infty }^{1-%
\frac{2}{\gamma }}+\Vert y\Vert _{s,t,\infty }^{1-%
\frac{1}{\gamma }}\right] (t-s)\,.  \label{e.2.4}
\end{equation}%
Since ${\alpha }>\frac{1}{{\beta }}-1$, we have ${\alpha }>{\alpha }{\beta }%
>1-{\beta }$. Thus $1-\alpha <\beta $. From Remark 1.1, we know that $y\in
C^{\beta }([0,T])$, for any $T>0$. A fractional
integration by parts formula (see Z\"ahle \cite{Za})
yields
\begin{equation}
\int_{s}^{t}y_{u}^{1-\frac{1}{\gamma }}d\varphi (u)=(-1)^{-\alpha
}\int_{s}^{t}D_{s+}^{\alpha }y_{u}^{1-\frac{1}{\gamma }}D_{t-}^{1-\alpha
}\varphi _{t-}(u)du\,,  \label{f3}
\end{equation}%
where $\varphi _{t-}(u)=\varphi (u)-\varphi (t)$, and $D_{s+}^{\alpha }$ and
$D_{t-}^{1-\alpha }$ denote the left and right-sided fractional derivatives
of orders $\alpha $ and $1-\alpha $, respectively (see \cite{SKM}), defined
by%
\begin{equation}
D_{s+}^{\alpha }y_{u}^{1-\frac{1}{\gamma }}=\frac{1}{\Gamma (1-\alpha )}%
\left( \frac{y_{u}^{1-\frac{1}{\gamma }}}{(u-s)^{\alpha }}+\alpha
\int_{s}^{u}\frac{y_{u}^{1-\frac{1}{\gamma }}-y_{r}^{1-\frac{1}{\gamma }}}{%
(u-r)^{\alpha +1}}dr\right) ,  \label{f1}
\end{equation}%
and%
\begin{equation}
D_{t-}^{1-\alpha }\varphi _{t-}(u)=\frac{(-1)^{1-\alpha }}{\Gamma (\alpha )}%
\left( \frac{\varphi (u)-\varphi (t)}{(t-u)^{1-\alpha }}+(1-\alpha
)\int_{u}^{t}\frac{\varphi (u)-\varphi (r)}{(r-u)^{2-\alpha }}dr\right) .
\label{f2}
\end{equation}
From (\ref{f1}), and using the H\"{o}lder continuity of $y$ we obtain
\begin{eqnarray}
|D_{s+}^{\alpha }y_{u}^{1-\frac{1}{\gamma }}| &\leq &C\left( \Vert y\Vert
_{s,t,\infty }^{1-\frac{1}{\gamma }}(u-s)^{-\alpha }+\int_{s}^{u}\frac{%
|y_{u}^{1-\frac{1}{\gamma }}-y_{r}^{1-\frac{1}{\gamma }}|}{(u-r)^{\alpha +1}}%
dr\right)  \notag \\
&\leq &C\left( \Vert y\Vert _{s,t,\infty }^{1-\frac{1}{\gamma }%
}(u-s)^{-\alpha }+\int_{s}^{u}\frac{|y_{u}-y_{r}|^{1-\frac{1}{\gamma }}}{%
(u-r)^{\alpha +1}}dr\right)  \notag \\
&\leq &C\left( \Vert y\Vert _{s,t,\infty }^{1-\frac{1}{\gamma }%
}(u-s)^{-\alpha }+\Vert y\Vert _{s,t,\beta }^{1-\frac{1}{\gamma }%
}\int_{s}^{u}(u-r)^{\beta (1-\frac{1}{\gamma })-\alpha -1}dr\right)  \notag
\\
&\leq &C\left( \Vert y\Vert _{s,t,\infty }^{1-\frac{1}{\gamma }%
}(u-s)^{-\alpha }+\Vert y\Vert _{s,t,\beta }^{1-\frac{1}{\gamma }%
}(u-s)^{\beta (1-\frac{1}{\gamma })-\alpha }\right) ,  \label{e.2.5}
\end{eqnarray}%
where and in what follows, $C$ denotes a generic constant depending on $%
\alpha $, $\beta $ and $T$. On the other hand, \ from (\ref{f2}) we have
\begin{equation}
|D_{t-}^{1-\alpha }\varphi _{t-}(u)|\leq C\Vert \varphi \Vert _{0,T,\beta
}(t-u)^{\alpha +\beta -1}.  \label{e.2.6}
\end{equation}%
Substituting (\ref{e.2.5}) and (\ref{e.2.6}) into (\ref{f3}) yields
\begin{eqnarray}
\left| \int_{s}^{t}y_{u}^{1-\frac{1}{\gamma }}d\varphi (u)\right| &\leq
&C\int_{s}^{t}\left( \Vert y\Vert _{s,t,\infty }^{1-\frac{1}{\gamma }%
}(u-s)^{-\alpha }+\Vert y\Vert _{s,t,\beta }^{1-\frac{1}{\gamma }%
}(u-s)^{\beta (1-\frac{1}{\gamma })-\alpha }\right)  \notag \\
&&\times \Vert \varphi \Vert _{0,T,\beta }(t-u)^{\alpha +\beta -1}du  \notag
\\
&\leq &C\Vert \varphi \Vert _{0,T,\beta }  \notag \\
&&\times \left( \Vert y\Vert _{s,t,\infty }^{1-\frac{1}{\gamma }%
}(t-s)^{\beta }+\Vert y\Vert _{s,t,\beta }^{1-\frac{1}{\gamma }}(t-s)^{\beta
(2-\frac{1}{\gamma })}\right) \,.  \label{f4}
\end{eqnarray}%
Substituting (\ref{f4}) and (\ref{e.2.4}) into (\ref{f5}) we obtain
\begin{eqnarray*}
|y_{t}-y_{s}| &\leq &K_{T}\gamma \left[
\Vert y\Vert _{s,t,\infty }^{1-%
\frac{2}{\gamma }}+\Vert y\Vert _{s,t,\infty }^{1-%
\frac{1}{\gamma }}\right](t-s)+C\gamma \Vert \varphi \Vert _{0,T,\beta } \\
&&\times \left( \Vert y\Vert _{s,t,\infty }^{1-\frac{1}{\gamma }%
}(t-s)^{\beta }+\Vert y\Vert _{s,t,\beta }^{1-\frac{1}{\gamma }}(t-s)^{\beta
(2-\frac{1}{\gamma })}\right) .
\end{eqnarray*}%
Consequently, using the estimate $x^{1-\frac{1}{\gamma }}\leq 1+x$ for all $%
x>0$, we obtain
\begin{eqnarray*}
\Vert y\Vert _{s,t,\beta } &\leq &K_{T}\gamma \left[
\Vert y\Vert _{s,t,\infty }^{1-%
\frac{2}{\gamma }}+\Vert y\Vert _{s,t,\infty }^{1-%
\frac{1}{\gamma }}\right](t-s)^{1-\beta }+C\gamma \Vert \varphi \Vert _{0,T,\beta }
\\
&&\times \left( \Vert y\Vert _{s,t,\infty }^{1-\frac{1}{\gamma }}+(1+\Vert
y\Vert _{s,t,\beta })(t-s)^{\beta (1-\frac{1}{\gamma })}\right) \,,
\end{eqnarray*}%
which implies
\begin{eqnarray*}
&& \left[ 1-C\gamma \Vert \varphi \Vert _{0,T,\beta }(t-s)^{\beta (1-\frac{1}{%
\gamma })}\right] \Vert y\Vert _{s,t,\beta } 
 \leq  K_{T}\gamma \left[
\Vert y\Vert _{s,t,\infty }^{1-%
\frac{2}{\gamma }}+\Vert y\Vert _{s,t,\infty }^{1-%
\frac{1}{\gamma }}\right] \\
&&  \qquad \qquad \times (t-s)^{1-\beta }  
  +C\gamma \Vert \varphi \Vert _{0,T,\beta }\left( \Vert
y\Vert _{s,t,\infty }^{1-\frac{1}{\gamma }}+(t-s)^{\beta (1-\frac{1}{\gamma }%
)}\right) \,.
\end{eqnarray*}%
Suppose that $\Delta $ satisfies
\begin{equation}
\Delta \leq \left( \frac{1}{2C\gamma \Vert \varphi \Vert _{0,T,\beta }}%
\right) ^{\frac{\gamma }{\beta (\gamma -1)}}.  \label{d}
\end{equation}%
Then for all $s,t\in \lbrack 0,T]$, $s\leq t$, such that $t-s\leq \Delta $,
we have
\begin{equation*}
\Vert y\Vert _{s,t,\beta }\leq 2K_{T}\gamma \left[
\Vert y\Vert _{s,t,\infty }^{1-%
\frac{2}{\gamma }}+\Vert y\Vert _{s,t,\infty }^{1-%
\frac{1}{\gamma }}\right](t-s)^{1-\beta }+2C\gamma \Vert \varphi \Vert _{0,T,\beta
}\Vert y\Vert _{s,t,\infty }^{1-\frac{1}{\gamma }}+1\,,
\end{equation*}%
and this implies
\begin{eqnarray}
\Vert y\Vert _{s,t,\infty } &\leq &|y_{s}|+\Vert y\Vert _{s,t,\beta
}(t-s)^{\beta }  \notag \\
&\leq &|y_{s}|+2K_T\gamma \left[
\Vert y\Vert _{s,t,\infty }^{1-%
\frac{2}{\gamma }}+\Vert y\Vert _{s,t,\infty }^{1-%
\frac{1}{\gamma }}\right]
(t-s)\notag\\
&&+2C\gamma \Vert \varphi \Vert _{0,T,\beta }\Vert y\Vert _{s,t,\infty
}^{1-\frac{1}{\gamma }}(t-s)^{\beta }+(t-s)^{\beta }.  \notag
\end{eqnarray}%
Using again the inequality $x^{\alpha }\le 1+x$ for all $x>0$ and ${\alpha }\in
(0,1)$, we have
\begin{eqnarray*}
\Vert y\Vert _{s,t,\infty } &\leq &|y_{s}|+4K_{T}\gamma \left( 1+\Vert
y\Vert _{s,t,\infty }\right) (t-s) \\
&&\quad +2C\gamma \Vert \varphi \Vert _{0,T,\beta }\left( 1+\Vert y\Vert
_{s,t,\infty }\right) (t-s)^{\beta }+(t-s)^{\beta },
\end{eqnarray*}%
which can be written as
\begin{eqnarray}
&&\Vert y\Vert _{s,t,\infty }\left( 1-2C\gamma \Vert \varphi \Vert
_{0,T,\beta }(t-s)^{\beta }-4K_{T}\gamma (t-s)\right)  \notag \\
&\leq &|y_{s}|+4K_{T}\gamma (t-s)+2(t-s)^{\beta }.  \label{f6}
\end{eqnarray}%
Now we choose $\Delta $ such that
\begin{equation}
\Delta =\left( \frac{1}{2C\gamma \Vert \varphi \Vert _{0,T,\beta }}\right) ^{%
\frac{\gamma }{\beta (\gamma -1)}}\wedge \left( \frac{1}{16K_{T}\gamma }%
\right) \wedge \left( \frac{1}{8C\gamma \Vert \varphi \Vert _{0,T,\beta }}%
\right) ^{\frac{1}{\beta }}\,.  \label{e}
\end{equation}%
Then, for all $s,t\in \lbrack 0,T]$, $s<t$, such that $t-s\leq \Delta $, the
inequality (\ref{f6}) implies
\begin{equation}
\Vert y\Vert _{s,t,\infty }\leq 2|y_{s}|+C_{{\gamma },{\beta },T}\,,
\label{e.2.13}
\end{equation}%
where $C_{{\gamma },{\beta },T}=8K_{T}{\gamma }T+4T^{\beta }$. Take $n=[%
\frac{T}{\Delta }]+1$ (where $[a]$ denotes the largest integer bounded by $a$%
). Divide the interval $[0,T]$ into $n$ subintervals. Applying the
inequality (\ref{e.2.13}) for $s=0$ and $t=\Delta $, we have for all $t\in
\lbrack 0,\Delta ]$
\begin{equation}
\Vert y\Vert _{0,t,\infty }\leq 2|y_{0}|+C_{{\gamma },{\beta },T}\,\ \,.
\label{e.2.14}
\end{equation}%
Applying the inequality (\ref{e.2.14}) on the intervals $[\Delta ,2\Delta
], \dots,[(n-2)\Delta ,(n-1)\Delta ],[(n-1)\Delta ,T]$ recursively, we obtain
\begin{eqnarray*}
\Vert y\Vert _{0,T,\infty } &\leq &2^{n}|y_{0}|+2^{n-1}C_{{\gamma },{\beta }%
,T}+\cdots +C_{{\gamma },{\beta },T} \\
&\leq &2^{\left[ \frac{T}{{\Delta }}\right] +1}(|y_{0}|+C_{{\gamma },{\beta }%
,T}) \\
&\leq &2^{T\left( 2C\gamma \Vert \varphi \Vert _{0,T,\beta }\right) ^{\frac{%
\gamma }{\beta (\gamma -1)}}\vee \left( 16K_T\gamma \right) \vee \left(
8C\gamma \Vert \varphi \Vert _{0,T,\beta }\right) ^{\frac{1}{\beta }%
}+1}(|y_{0}|+C_{{\gamma },{\beta },T})\,.
\end{eqnarray*}%
Therefore, we obtain
\begin{equation*}
\Vert x\Vert _{0,T,\infty }\leq C_{1,{\gamma },{\beta },T}(|x_{0}|+1)\exp
\left\{ C_{2,{\gamma },{\beta },T}\left( 1+\Vert \varphi \Vert _{0,T,\beta
}^{\frac{\gamma }{\beta (\gamma -1)}}\right) \right\} ,
\end{equation*}%
which concludes the proof of the theorem.
\end{proof}

\setcounter{equation}{0}

\section{Singular equations driven by fBm}

Let $(B_{t}^{H},t\geq 0)$ be a fractional Brownian motion with Hurst
parameter $H\in (1/2,1)$, defined in a complete probability space $(\Omega ,%
\mathcal{F},P)$. Namely, $(B_{t}^{H},t\geq 0)$ is a mean zero Gaussian
process with covariance
\begin{equation}
\mathbb{E}(B_{t}^{H}B_{s}^{H})=R_{H}(s,t)=\frac{1}{2}\left(
t^{2H}+s^{2H}-|t-s|^{2H}\right) \,.  \label{cov}
\end{equation}%
We are interested in the following singular stochastic differential equation
\begin{equation}
X_{t}=x_{0}+\int_{0}^{t}f(s,X_{s})ds+B_{t}^{H}\,,  \label{a.2}
\end{equation}%
where $x_{0}>0$, and the function $f(s,x)$ has a singularity at $x=0$ and
satisfies the assumptions (i) to (iii). As an immediate consequence of \
Theorem \ref{t1} we have the following result.

\begin{theorem}
\label{t2}Under the assumptions (i)-(iii), there is a unique pathwise solution $%
X=(X_{t},t\geq 0)$ to Equation (\ref{a.2}), such that $X_{t}>0$ t almost
surely on $[0,\infty )$ and for any $T>0$, $\Vert X\Vert _{0,T,\infty }\in
L^{p}(\Omega )$, for all $p>0$.
\end{theorem}

\begin{proof}
Fix $\beta \in (\frac{1}{2},H)$ and $T>0$. Applying Theorem \ref{t1}, we
obtain that there is a unique pathwise solution $X=(X_{t},t\geq 0)$ to
Equation (\ref{a.2}), such that $X_{t}>0$ almost surely on $[0,\infty )$ and
\begin{equation}
\Vert X\Vert _{0,T,\infty }\leq C_{1,{\gamma },{\beta },T}(|x_{0}|+1)\exp
\left\{ C_{2,{\gamma },{\beta },T}\left( 1+\Vert B^{H}\Vert _{0,T,\beta }^{%
\frac{\gamma }{\beta (\gamma -1)}}\right) \right\} \,.  \label{e.4.3}
\end{equation}%
If we choose ${\gamma }$ such that $\gamma >\frac{2\beta }{2\beta -1}$, then
$\frac{\gamma }{\beta (\gamma -1)}<2$, and by Fernique's theorem (see \cite%
{Fe}), we obtain
\begin{equation}
\mathbb{E}(e^{C\Vert B^{H}\Vert _{0,T,\beta }^{\frac{\gamma }{\beta (\gamma
-1)}}})<\infty ,  \label{e.4.2}
\end{equation}%
for all $C>0$, which implies that $\mathbb{E}(\Vert X\Vert _{0,T,\infty
}^{p})<\infty $ for all $p\geq 1$.
\end{proof}

Theorem \ref{t2} implies the existence of a unique solution to the following
stochastic differential equation with non Lipschitz diffusion coefficient:
\begin{equation}
Y_{t}=y_{0}+\int_{0}^{t}f(s,Y_{s})ds+\int_{0}^{t}\sqrt{Y_{s}}dB_{s}^{H},
\label{young}
\end{equation}%
where $y_{0}$ is a positive constant and $f$ is a nonnegative continuous
function satisfying the following conditions:
\begin{description}
\item[(a)] There exists $x_1>0$ such that $f(t,x)\ge g(t)$ for all $t>0$ and $x\in(0,x_1)$, where $g$ is a continuous function such that $g(t)>0$ if $t>0$.
\item[(b)] $f(t,x) \ge x \partial _{x}f(t,x)$ for all $t>0$ and $x>0$.
\item[(c)] $f(t,x) \le h(t) (x+1)$ for all $t\ge 0$ and $x>0$, where $h$ is a nonnegative locally bounded function.
\end{description}
The term $\sqrt {Y_s}$ appears in a fractional version of the CIR process in financial
mathematics (see \cite{CIR}) and cannot be treated directly
by the approaches in Lyons \cite{Ly},
Nualart and R\u a\c scanu \cite{NR}, since function $%
g(x)=\sqrt{x}$ does no satisfy the \ usual \
Lipschitz conditions commonly imposed.  We make the change of variables
$X_{t}=2\sqrt{Y_{t}}$. Then, from the
chain rule for the Young integral, it follows that a positive stochastic
process $Y=(Y_{t},t\geq 0)$ satisfies (\ref{young}) if and  only if $X_{t}$
satisfies the following equation:
\begin{equation}
X_{t}=2\sqrt{y_{0}}+\int_{0}^{t}\frac{2f(s,X_{s})}{X_{s}}ds+B_{t}^{H}.
\label{aa}
\end{equation}%
Let $f_1(t,x)=2f(t,x)x^{-1}$. Then $f_1(t,x)$ satisfies all assumptions
(i)-(iii), and hence from Theorem \ref{t2}, we know that there exists a
unique positive solution $X_{t}$ to equation (\ref{aa}) with all positive
moments. So $Y_{t}=X_{t}^{2}/4$ is the unique positive solution \ to
Equation (\ref{young}), and it has finite moments of all orders.

\medskip
The next result  states the scaling property of the solution to Equation (%
\ref{a.2}), when the coefficient $f(s,x)$ satisfies some homogeneity
condition on the variable $x$.

\begin{proposition}
(Scaling Property) \ We denote by $Eq(x_{0},f)$ the equation (\ref{a.2}).
Suppose that $x_{0}>0$, and $f(t,x)$ satisfies assumptions (i)-(iii), and $%
f(t,x)$ is homogeneous, that is, $f(st,yx)=s^{m}y^{n}f(t,x)$ for some
constants $m,n$. Then, the process $\left( a^{H}X_{\frac{t}{a}},t\geq
0\right) $ has the same law as the solution to the Equation $%
Eq(a^{H}x_{0},a^{H-nH-m-1}f)$.
\end{proposition}

\begin{proof}
For each $a>0$, we know that $\{a^{-H}B_{at}^{H},t\geq 0\}$ is a fractional
Brownian motion. We denote $X_{a,t}$ the solution to the following equation:
\begin{equation*}
X_{a,t}=x_{0}+\int_{0}^{t}f(s,X_{a,s})ds+a^{-H}B_{at}^{H}.
\end{equation*}%
So $\left( X_{t},t\geq 0\right) $ (the solution to $Eq(x_{0},f)$) has the
same distribution as $\left( X_{a,t},t\geq 0\right) $. Then%
\begin{eqnarray*}
a^{H}X_{a,\frac{t}{a}} &=&a^{H}x_{0}+\int_{0}^{\frac{t}{a}%
}a^{H}f(s,X_{a,s})ds+\ B_{t}^{H} \\
&=&a^{H}x_{0}+\int_{0}^{t}a^{H-1-m-nH}f(r,a^{H}X_{a,\frac{r}{a}})dr+\
B_{t}^{H},
\end{eqnarray*}%
which implies the result.
\end{proof}

As an example, we can consider the function $f(t,x)=s^{\gamma }x^{-\alpha }$%
, where $\alpha >\frac{1}{H}-1$, and $\gamma >0$. \ Then, if $(X_{t},t\geq 0)
$ is the solution to Equation%
\begin{equation*}
X_{t}=x_{0}+\int_{0}^{t}s^{\gamma }X_{s}^{-\alpha }ds+B_{t}^{H}
\end{equation*}
(\ref{a.2}), then\ $\left( a^{H}X_{\frac{t}{a}},t\geq 0\right) $ has the
same law as the solution to the Equation $\ $%
\begin{equation*}
X_{t}=a^{H}x_{0}+a^{H-\alpha H-\gamma -1}\int_{0}^{t}s^{\gamma
}X_{s}^{-\alpha }ds+B_{t}^{H}.
\end{equation*}

\subsection{Absolute continuity of the law of the solution}

In this subsection we will  apply the Malliavin calculus to the solution to
Equation    (\ref{a.2}) in order to study the absolute continuity of the law
of the solution at a fixed time $t>0$. \ We will first make some
preliminaries on the Malliavin calculus for the fractional Brownian motion,
and we refer to Decreusefond and \"{U}st\"{u}nel \cite{DU}, Nualart \cite{nualart}
and Saussereau and Nualart \cite{SN} for a more complete treatment of this
topic.

Fix a time interval $[0,T]$. Denote by $\mathcal{E}$ the set of real valued
step functions on   $[0,T]$ and let $\mathcal{H}$ be the Hilbert space
defined as the closure of $\mathcal{E}$ with respect to the scalar product $%
\langle \mathbf{1}_{[0,t]},\mathbf{1}_{[0,s]}\rangle _{\mathcal{H}%
}=R_{H}(t,s)$, where $R_{H}$ is the covariance function of the fBm, given in
(\ref{cov}). We know that
\begin{eqnarray}
R_{H}(t,s) &=&\alpha _{H}\int_{0}^{t}\int_{0}^{s}|r-u|^{2H}dudr  \notag \\
&=&\int_{0}^{t\wedge s}K_{H}(t,r)K_{H}(s,r)dr,  \notag
\end{eqnarray}%
where $K_{H}(t,s)=c_{H}s^{\frac{1}{2}-H}\int_{s}^{t}(u-s)^{H-\frac{3}{2}%
}u^{H-\frac{1}{2}}du\mathbf{1}_{\{s<t\}}$ with $c_{H}=\sqrt{\frac{H(2H-1)}{%
B(2-2H,H-\frac{1}{2})}}$ and $B$ denotes the Beta function, and $\alpha
_{H}=H(2H-1)$.  In general, for any   $\varphi ,\psi \in \mathcal{E}$ we
have
\begin{equation*}
\langle \varphi ,\psi \rangle _{\mathcal{H}}=\alpha
_{H}\int_{0}^{T}\int_{0}^{T}|r-u|^{2H-2}\varphi _{r}\psi _{u}dudr\,\,.
\end{equation*}%
The mapping $\mathbf{1}_{[0,t]}$$\longmapsto B_{t}^{H}$ can be extended to
an isometry between $\mathcal{H}$ and the Gaussian space $\mathcal{H}_{1}$
spanned by $B^{H}$. We denote this isometry by $\varphi \longmapsto
B^{H}(\varphi )$.

We consider the operator $K_{H}^{\ast }:\mathcal{E}\rightarrow L^{2}(0,T)$
defined by $\ $
\begin{equation}
\left( K_{H}^{\ast }\varphi \right) (s)=\int_{s}^{T}\varphi (t)\frac{%
\partial K_{H}}{\partial t}(t,s)dt.  \label{h1}
\end{equation}
Notice that $(K_{H}^{\ast }(\mathbf{1}_{[0,t]}))(s)=K_{H}(t,s)\mathbf{1}%
_{[0,t]}(s)$. For any $\varphi ,\psi \in \mathcal{E}$ we have%
\begin{equation}
\langle \varphi ,\psi \rangle _{\mathcal{H}}=\langle K_{H}^{\ast }\varphi
,K_{H}^{\ast }\phi \rangle _{L^{2}(0,T)}=\mathbb{E}(B^{H}(\varphi
)B^{H}(\phi )),  \label{h2}
\end{equation}
and   $K_{H}^{\ast }$ provides an isometry between the Hilbert space $%
\mathcal{H}$ and a closed subspace of $L^{2}(\left[ 0,T\right] )$. We denote
$K_{H}:L^{2}(\left[ 0,T\right] )\rightarrow \mathcal{H}_{H}:=K_{H}(L^{2}(%
\left[ 0,T\right] ))$ the operator defined by $(K_{H}h)(t):=%
\int_{0}^{t}K_{H}(t,s)h(s)ds$. The space $\mathcal{H}_{H}$ is the fractional
version of the Cameron-Martin space. Finally, we denote by $R_{H}=K_{H}\circ
K_{H}^{\ast }:\mathcal{H}\rightarrow \mathcal{H}_{H}$ the operator $%
R_{H}\varphi =\int_{0}^{\cdot }K_{H}(\cdot ,s)(K_{H}^{\ast }\varphi )(s)ds$.
We remark that for any $\varphi \in \mathcal{H}$, $R_{H}\varphi $ is
H\"{o}lder continuous of order $H$.

\medskip
If we assume that $\Omega $ is the canonical probability space $C_{0}([0,T])$,
equipped with the Borel $\sigma $-field and the probability $P$ is the law
of the fBm. Then, the injection $R_{H}:\mathcal{H}\rightarrow $$\Omega $
embeds $\mathcal{H}$ densely into $\Omega $ and $(\Omega ,\mathcal{H},$ $P)$
is an abstract Wiener space in the sense of Gross.
In the sequel we will make this assumption  on the underlying probability space.

Let $\mathcal{S}$ be the space of   smooth and cylindrical random
variables of the form%
\begin{equation}
F=f(B^{H}(\varphi _{1}),\ldots ,B^{H}(\varphi _{n})),  \label{g1}
\end{equation}%
where $f\in C_{b}^{\infty }(\mathbb{R}^{n})$ ($f$ and all its partial
derivatives are bounded). For a random variable $F$ of the form \ (\ref{g1})
we define its Malliavin derivative as the $\mathcal{H}$-valued random variable%
\begin{equation*}
DF=\sum_{i=1}^{n}\frac{\partial f}{\partial x_{i}}(B^{H}(\varphi
_{1}),\ldots ,B^{H}(\varphi _{n}))\varphi _{i}\text{. }
\end{equation*}%
We denote by $\mathbb{D}^{1,2}$ the Sobolev space defined as the completion
of the class $\mathcal{S}$, with respect to the norm%
\begin{equation*}
\left\| F\right\| _{1,2}=\left[ \mathbb{E}(F^{2})+\mathbb{E}\left( \left\|
DF\right\| _{\mathcal{H}}^{2}\right) \right] ^{1/2}.
\end{equation*}%
The basic criterion for the existence of densities (see Bouleau and Hirsch %
\cite{BH}), says that if \ $F\in \mathbb{D}^{1,2}$, and \ $\left\|
DF\right\| _{\mathcal{H}}>0$ almost surely, then the law of $F$ has a
density with respect to the Lebesgue measure on the real line. Using this
criterion we can show the following result.

\begin{theorem}
Suppose that $f$ satisfies the assumptions (i)-(iii). Let $X_{t}$ be the
solution to Equation (\ref{a.2}). Then for any $t\geq 0$, $X_{t}\in \mathbb{D%
}^{1,2}$. Furthermore, \ for any $t>0$ the law of $X_{t}$ is absolutely
continuous with respect to the Lebesgue measure on $\mathbb{R}$.
\end{theorem}

\begin{proof}
Fix a time interval $[0,T]$, and let $\beta \in (\frac{1}{2},H)$.  We want
to \ compute the directional derivative $\left\langle DX_{t},\varphi
\right\rangle _{\mathcal{H}}$, for some $\varphi \in \mathcal{H}$. The
function  $h=R_{H}\varphi $ belongs to $C^{\beta }([0,T])$ and   $h_{0}=0$.
Taking into account the embedding given by $R_{H}:\mathcal{H}\rightarrow $$%
\Omega $ mentioned before, we will have%
\begin{equation}
\left\langle DX_{t},\varphi \right\rangle _{\mathcal{H}}=\frac{%
dX_{t}^{\epsilon }}{d\epsilon }|_{\epsilon =0},  \label{h0}
\end{equation}%
where $X_{t}^{\epsilon }$ is the solution to   the following equation
\begin{equation}
X_{t}^{\epsilon }=x_{0}+\int_{0}^{t}f(s,X_{s}^{\epsilon
})ds+B_{t}^{H}+\epsilon h_{t},  \label{h}
\end{equation}%
$t\in \lbrack 0,T]$, where $\epsilon $ $\in \lbrack 0,1]$.

From Theorem \ref{t2}, \ it follows that there is a constant $C$ independent
of ${\varepsilon }$ such that
\begin{equation*}
\mathbb{E}\left( \sup_{0\leq t\leq T}|X^{\epsilon }|_{t}^{p}\right) \leq
C<\infty \,,
\end{equation*}%
for all $p\geq 1$.  From equations (\ref{a.2}) and (\ref{h}), we
deduce
\begin{equation}
X_{t}^{\epsilon }-X_{t}=\int_{0}^{t}(f(s,X_{s}^{\epsilon
})-f(s,X_{s}))ds+\epsilon h_{t}.  \label{bb}
\end{equation}%
By using Taylor expansion, the equation (\ref{bb}) becomes:
\begin{equation}
X_{t}^{\epsilon }-X_{t}=\int_{0}^{t}\Theta _{s}(X_{s}^{\epsilon
}-X_{s})ds+\epsilon h_{t}\,,  \label{e.4.4}
\end{equation}%
where $\Theta _{s}=\partial _{x}f(s,X_{s}+\theta _{s}(X_{s}^{\epsilon
}-X_{s}))$ for some $\theta _{s}^{\epsilon }$ between $0$ and $1$.
By using   (\ref{h1})   the
solution to equation (\ref{e.4.4}) is given by
\begin{eqnarray}
X_{t}^{\epsilon }-X_{t} &=&\epsilon \int_{0}^{t}\exp \left(
\int_{s}^{t}\Theta _{r}dr\right) d(R_{H}\varphi )(s)  \notag \\
&=&\epsilon \int_{0}^{t}\exp \left( \int_{s}^{t}\Theta _{r}dr\right) \left(
\int_{0}^{s}\frac{\partial K_{H}(s,u)}{\partial s}(K_{H}^{\ast }\varphi
)(u)du\right) ds  \notag \\
&=&\epsilon \int_{0}^{t}\left( \int_{u}^{t}\exp \left( \int_{s}^{t}\Theta
_{r}dr\right) \frac{\partial K_{H}(s,u)}{\partial s}ds\right) (K_{H}^{\ast
}\varphi )(u)du.  \notag
\end{eqnarray}%
Using (\ref{h1}) and (\ref{h2}) we can write%
\begin{eqnarray*}
X_{t}^{\epsilon }-X_{t} &=&\epsilon \int_{0}^{t}(K_{H}^{\ast }\varphi
)(u)\left( K_{H}^{\ast }\left( \exp \left( \int_{\cdot }^{t}\Theta
_{r}dr\right) \right) \right) (u)du \\
&=&\epsilon \left\langle \varphi ,\exp \left( \int_{\cdot }^{t}\Theta
_{r}dr\right) \right\rangle _{\mathcal{H}} \\
&=&\epsilon \alpha _{H}\int_{0}^{t}\int_{0}^{t}\varphi (s)\exp \left(
\int_{u}^{t}\Theta _{r}dr\right) |s-u|^{2H-2}duds
\end{eqnarray*}%
Since $\partial _{x}f(t,x)$ is continuous and $\partial _{x}f(t,x)\leq 0\ $\
for all $t>0$ and $x>0$, we have $\exp \left( \int_{u}^{t}\Theta
_{r}dr\right) \leq 1$. \ As a consequence,%
\begin{eqnarray*}
\lim_{\epsilon \rightarrow 0}\frac{X_{t}^{\epsilon }-X_{t}}{\epsilon }
&=&\alpha _{H}\int_{0}^{t}\int_{0}^{t}\varphi (s)\exp \left(
\int_{u}^{t}\partial _{x}f(r,X_{r})dr\right) |s-u|^{2H-2}duds \\
&=&\left\langle \varphi ,\exp \left( \int_{\cdot }^{t}\partial
_{x}f(r,X_{r})dr\right) \mathbf{1}_{[0,t]}\right\rangle _{\mathcal{H}},
\end{eqnarray*}%
where the limit holds almost surely and in $L^{2}(\Omega )$.  Then, taking
into account (\ref{h0}), by the results of Sugita \cite{Su}, we have $%
X_{t}\in \mathbb{D}^{1,2}$, and
\begin{equation}
DX_{t}=\exp \left( \int_{\cdot }^{t}\partial _{x}f(r,X_{r})dr\right) \mathbf{%
1}_{[0,t]}.  \label{a1}
\end{equation}%
Finally,%
\begin{eqnarray*}
 \left\| DF\right\| _{\mathcal{H}}^{2}&=&\alpha
_{H}\int_{0}^{t}\int_{0}^{t}\exp \left( \int_{s}^{t}\partial
_{x}f(r,X_{r})dr\right)   \\
&&\times  \exp \left( \int_{u}^{t}\partial
_{x}f(r,X_{r})dr\right) |s-u|^{2H-2}duds>0.
\end{eqnarray*}%
\
\end{proof}

In the next proposition we will show the existence
of negative moments for the function  $%
f(t,x)=\frac{K}{x}$, where $K>0$ is a constant.
The proof is based again on the techniques of Malliavin calculus.

\begin{proposition}
Let $\left( X_{t},t\geq 0\right) $ be the solution to the equation%
\begin{equation}
X_{t}=x_{0}+\int_{0}^{t}\frac{K}{X_{s}}ds+B_{t}^{H}.  \label{b}
\end{equation}%
Then, for all $p\geq 1$ and $t\leq \left( \frac{K}{(p+1)H}\right) ^{\frac{1}{%
2H-1}}$ we have $E(X_{t}^{-p})\leq x_{0}^{-p}$.
\end{proposition}

\begin{proof}
Obviously the function  $f(t,x)=\frac{K}{x}$ satisfies all the conditions
(i)-(iii).  From Equation (\ref{a1}), we have
\begin{equation*}
D_{s}X_{t}=-\int_{0}^{t}\frac{K}{X_{r}^{2}}D_{s}X_{r}dr+\mathbf{1}%
_{[0,t]}(s).
\end{equation*}%
So,
\begin{equation*}
D_{s}X_{t}=\exp \{-\int_{s}^{t}\frac{K}{X_{r}^{2}}dr\}\mathbf{1}_{[0,t]}(s).
\end{equation*}%
For any fixed $p\geq 1$, we construct the family of functions $\ \varphi
_{\epsilon }(x)=\frac{1}{(\epsilon +x)^{p}}$, $x>0$. Then $\varphi
_{\epsilon }\uparrow x^{-p}$, as $\epsilon \downarrow 0$. For each $\epsilon
>0$, $\varphi _{\epsilon }$ is a bounded continuously differentiable
function and   all its derivatives are bounded.

By  the chain rule we obtain,
\begin{eqnarray}
\varphi _{\epsilon }(X_{t}) &=&\varphi _{\epsilon
}(x_{0})+\int_{0}^{t}\varphi _{\epsilon }^{\prime }(X_{s})\frac{K}{x_{s}}%
ds+\int_{0}^{t}\varphi _{\epsilon }^{\prime \prime }(X_{s})dB_{s}^{H}  \notag
\\  \notag
&=&\varphi _{\epsilon }(x_{0})-p\int_{0}^{t}\frac{K}{X_{s}(\epsilon
+X_{s})^{p+1}}ds \\
&& -p\int_{0}^{t}\frac{1}{(\epsilon +X_{s})^{p+1}}dB_{s}^{H}
\label{z2}
\end{eqnarray}%
Then, Proposition 5.3.2 in \cite{nualart} implies that%
\begin{eqnarray}
\int_{0}^{t}\frac{1}{(\epsilon +X_{s})^{p+1}}dB_{s}^{H} &=&\delta \left(
\frac{1}{(\epsilon +X_{s})^{p+1}}\mathbf{1}_{[0,t]}(s)\right)   \notag
 -(p+1)\alpha _{H} \\
 &&\times \int_{0}^{t}\int_{0}^{t}\frac{D_{r}X_{s}}{(\epsilon
+X_{s})^{p+2}}\ |s-r|^{2H-2}drds,  \label{z1}
\end{eqnarray}%
where $\delta $ is the divergence operator with respect to fractional
Brownian motion. Substituting (\ref{z1}) into (\ref{z2}) yields%
\begin{eqnarray*}
\varphi _{\epsilon }(X_{t}) &\leq &\varphi _{\epsilon }(x_{0})-p\int_{0}^{t}%
\frac{K}{(\epsilon +X_{s})^{p+2}}ds-p\delta \left( \frac{1}{(\epsilon
+X_{s})^{p+1}}\mathbf{1}_{[0,t]}(s)\right)  \\
&&+p(p+1)Ht^{2H-1}\int_{0}^{t}\frac{1}{(\epsilon +X_{s})^{p+2}}ds \\
&=&\varphi _{\epsilon }(x_{0})-p\int_{0}^{t}\frac{K-(p+1)Ht^{2H-1}}{%
(\epsilon +X_{s})^{p+2}}ds-p\delta \left( \frac{1}{(\epsilon +X_{s})^{p+1}}%
\mathbf{1}_{[0,t]}(s)\right) .
\end{eqnarray*}%
Fix some $t$, such that $K-(p+1)Ht^{2H-1}\geq 0$, that is $t\leq \left(
\frac{K}{(p+1)H}\right) ^{\frac{1}{2H-1}}$. Taking expectation on above
inequality, we have
\begin{equation*}
\mathbb{E}(\varphi _{\epsilon }(X_{t}))\leq \varphi _{\epsilon }(x_{0})\leq
\ x_{0}^{-p}.
\end{equation*}%
Let $\epsilon $ tends to $0$. Applying monotone convergence theorem, for any
fixed $p\geq 1$, we obtain
\begin{equation*}
\mathbb{E}(X_{t}^{-p})\leq \ x_{0}^{-p}\,.
\end{equation*}
\end{proof}

\end{document}